\documentclass[preprint,12pt,3p]{elsarticle}
\usepackage{verbatim}
\usepackage{amsmath,amsthm,amsfonts,amssymb}
\usepackage{algorithm}
\usepackage{algpseudocode}
\usepackage{cases}
\usepackage{hyperref}
\usepackage[normalem]{ulem}
\usepackage{xcolor,sectsty}
\definecolor{astral}{RGB}{46,116,181}
\subsectionfont{\color{astral}}
\sectionfont{\color{astral}}
\newtheorem{theorem}{Theorem}[section]
\newtheorem{lemma}[theorem]{Lemma}

\newtheorem{definition}[theorem]{Definition}
\newtheorem{example}[theorem]{Example}

\newtheorem{remark}[theorem]{Remark}
\newtheorem{thm}{Theorem}[section]
\newtheorem{ex}[thm]{Example}

\definecolor{darkslategray}{rgb}{0.18, 0.31, 0.31}
\definecolor{warmblack}{rgb}{0.0, 0.26, 0.26}
\journal{arXiv.org}
\newcommand{\nl}{\newline}

\newcommand{\R}{{\mathbb R}}
\newcommand{\C}{{\mathbb C}}

\newcommand{\mc}[1]{\mathcal {#1}}
\newcommand{\dg}{{\dagger}}
\newcommand{\n}{{*_N}}
\newcommand{\1}{{*_1}}
\newcommand{\m}{{*_M}}

\begin{document}

\begin{frontmatter}

\title{ \textcolor{warmblack}{\bf Reverse order law for the Moore-Penrose  inverses of
tensors }}

\author{Krushnachandra Panigrahy$^\dag$$^a$, Ratikanta Behera$^*$, and Debasisha Mishra$^\dag$$^b$}

\address{$^{*}$
Department of Mathematics and Statistics,\\
Indian Institute of Science Education and Research Kolkata,\\
 Nadia, West Bengal, India.\\
\textit{E-mail}: \texttt{ratikanta@iiserkol.ac.in}

                \vspace{.3cm}

                       $^{\dag}$ Department of Mathematics,\\
                        National Institute of Technology Raipur,\\
                        Raipur, Chhattisgarh, India.
                        \\\textit{E-mail$^a$}: \texttt{kcp.224\symbol{'100}gmail.com }
                        \\\textit{E-mail$^b$}: \texttt{dmishra\symbol{'100}nitrr.ac.in. }}

\begin{abstract}
\textcolor{warmblack}{Reverse order law for the Moore-Penrose
inverses of  tensors are  useful in the field of multilinear
algebra. In this paper, we first prove some more identities
involving the  Moore-Penrose inverse of tensors. We then obtain a
few necessary and sufficient conditions for the reverse order law
for the Moore-Penrose inverse of tensors via the Einstein product. }
\end{abstract}


\end{frontmatter}

\section{Introduction}
\label{sec1}

Higher-order generalizations of vectors and matrices are referred to
tensors, and have attracted tremendous interest in recent years
\cite{kolda,  Marcar, loan}. Let $\mathbb{C}^{I_1\times\cdots\times
I_N}$  be the set of  order $N$ and dimension $I_1 \times \cdots
\times I_N$ tensors over the complex field $\mathbb{C}$.  $\mc{A}
\in \mathbb{C}^{I_1\times\cdots\times I_N}$ is a multiway array with
$N$-th order tensor and $I_1, I_2, \cdots, I_N$ are dimensions of
the first, second,  $\cdots$ , $N$th way,
 respectively. Each entry of $\mc{A}$ is denoted by $a_{i_1...i_N}$.
The Einstein product \cite{ein} $ \mc{A}\n\mc{B} \in
\mathbb{C}^{I_1\times\cdots\times I_N \times J_1 \times\cdots\times
J_M }$ of tensors $\mc{A} \in \mathbb{C}^{I_1\times\cdots\times I_N
\times K_1 \times\cdots\times K_N }$ and $\mc{B} \in
\mathbb{C}^{K_1\times\cdots\times K_N \times J_1 \times\cdots\times
J_M }$   is defined by the operation $\n$ via
\begin{equation*}\label{Eins}
(\mc{A}\n\mc{B})_{i_1...i_Nj_1...j_M}
=\displaystyle\sum_{k_1...k_N}a_{{i_1...i_N}{k_1...k_N}}b_{{k_1...k_N}{j_1...j_M}},
\end{equation*}
The associative law of this tensor product holds. In the above
formula, if $\mc{B} \in \mathbb{C}^{K_1\times\cdots\times K_N}$,
then $\mc{A}\n\mc{B} \in \mathbb{C}^{I_1\times\cdots\times I_N}$ and
\begin{equation*}
(\mc{A}\n\mc{B})_{i_1...i_N} = \displaystyle\sum_{k_1...k_N}
a_{{i_1...i_N}{k_1...k_N}}b_{{k_1...k_N}}.
\end{equation*}
This product is  used in the study of the theory of relativity
\cite{ein} and in the area of continuum mechanics \cite{lai}. The
Einstein product $\1$  reduces to the standard matrix multiplication
as
$$(A\1B)_{ij}= \displaystyle\sum_{k=1}^{n} a_{ik}b_{kj},$$
for $A\in {\R}^{m \times n}$ and $B\in {\R}^{n \times l}$.

Brazell {\it et al.} \cite{BraliNT13} introduced the notion of the
ordinary tensor inverse,  as follows. A tensor $\mc{X} \in
\mathbb{C}^{I_1\times\cdots\times I_N \times I_1 \times\cdots\times
I_N}$ is called  the  {\it inverse} of $\mc{A}\in
\mathbb{C}^{I_1\times\cdots\times I_N \times I_1 \times\cdots\times
I_N}$  if it satisfies $\mc{A}\n\mc{X}=\mc{X}\n\mc{A}=\mc{I}$. It is
denoted  by $\mc{A}^{-1}$. If  $\mc{A}\in
\mathbb{C}^{I_1\times\cdots\times I_N \times I_1 \times\cdots\times
I_N }$ and $\mc{B}\in \mathbb{C}^{I_1\times\cdots\times I_N \times
I_1 \times\cdots\times I_N }$ are a pair of invertible tensors such
that their Einstein product $\mc{A}\n\mc{B}$ is also invertible,
then the reverse-order law for invertible  tensors $A$ and $B$ is
$$(\mc{A}\n\mc{B})^{-1} = \mc{B}^{-1}\n\mc{A}^{-1}.$$

This paper is concerned with the  reverse order law for the
Moore-Penrose inverses of  tensors via the Einstein product. Before
moving into the same, let us  recall the definition of the
Moore-Penrose inverse of a tensor which was introduced in
\cite{sun}, very recently.

\begin{definition}(Definition 2.2, \cite{sun})\label{defmpi}\\
Let $\mc{A} \in \mathbb{C}^{I_1\times\cdots\times I_N \times J_1
\times ... \times J_N}$. The tensor $\mc{X} \in
\mathbb{C}^{J_1\times\cdots\times J_N \times I_1 \times\cdots\times
I_N}$ satisfying the following four tensor equations:

\begin{enumerate}
\item[(1)] $\mc{A}\n\mc{X}\n\mc{A} = \mc{A};$
\item[(2)] $\mc{X}\n\mc{A}\n\mc{X} = \mc{X};$
\item[(3)] $(\mc{A}\n\mc{X})^* = \mc{A}\n\mc{X};$
\item[(4)] $(\mc{X}\n\mc{A})^* = \mc{X}\n\mc{A},$
\end{enumerate}
is called  the \textbf{Moore-Penrose inverse} of $\mc{A}$, and is
denoted by $\mc{A}^{\dg}$.
\end{definition}

Thereafter, the authors of \cite{bm,  weit2} further introduced
different generalized inverses of tensors   via the Einstein product
and   added a few more results to the same theory. But Jin {\it et
al.} \cite{bai} introduced  the Moore-Penrose inverse of  a tensor
using $t-$product.  However, the so-called reverse order law is not
necessarily true for any kind of generalized inverses. In
particular, Behera and Mishra \cite{bm} provided a characterization
of the reverse order law for $\{ 1 \}$-inverse of tensors (see
Theorem 2.16, \cite{bm}), and    obtained an example which shows
that the reverse order law for the Moore-Penrose inverses  of
tensors is not true in general (see Example 2.4, \cite{bm}). At the
last, they proposed the following open question:

 \vspace{-0.5cm}

$$\mbox{{\bf Question 1.} When does}~~  (\mc{A}\n\mc{B})^{\dg} = \mc{B}^{\dg} \n   \mc{A}^{\dg}  ?$$

 This is also called as {\it two term reverse order law}. The reverse order law for the Moore-Penrose inverses of a  tensor product yields a class of interesting problems that are fundamental in the theory of generalized inverses of tensors.
 The notion of the reverse order law for the Moore-Penrose inverses of  matrices has a long history.  Greville \cite{gre} studied first the above problem  but in the setting of rectangular matrices. Baskett and Katz \cite{bask} then discussed the same theory for $EP_r$ matrices where $A\in {\C}^{n\times n}$ of rank $r$ is called $EP_r$, if $A$ and $A^*$, the
conjugate transpose of $A$, have the same null spaces.  The reverse
order law was also studied
 for  other generalized inverses of matrices (see \cite{bar}, \cite{cao}, \cite{weir}, \cite{tian}  and references there
in). It was later carried forward by
  Bouldin \cite{boul} to  bounded linear operators with  closed range spaces. Hartwig \cite{har}
  provided necessary and sufficient conditions for holding of triple (or {\it three term)} reverse order law (i.e., $(ABC)^{\dg}=C^{\dg}B^{\dg}A^{\dg}$ where $A$, $B$ and $C$ are matrices).
  The study of this problem for generalized inverses in C*-algebras can be seen in the work by
  Cvetkovi´c-Ili´c and  Hartee \cite{cve} and Mosic and Djordjevic \cite{mosic}.
  While Deng \cite{deng} studied the same problem for the group invertible operators,
  Wang {\it et al.} \cite{wang} considered  for the Drazin invertible operators.
    The vast work on the reverse order law and its several multivariety extensions in different areas of mathematics in the literature  and the recent works in \cite{sun} and \cite{bm} motivate us
to study this problem in the framework of tensors.

 The main objective of this paper is to answer the above question and to do this, the paper is outlined as follows. In the next Section, we discuss some notations and definitions which are helpful in proving the main results.  Section 3 discusses the main results and  has two parts. In the first part,  we obtain several identities involving the Moore-Penrose inverse of tensors and the trace of a tensor. The second part contains
  a few necessary and sufficient conditions of the  reverse order law for the Moore-Penrose inverses of  tensors via the Einstein product. 

\section{Preliminaries}
For convenience, we first briefly explain some of the terminologies
which will be used here on wards.  We  refer to ${\R}^{m \times n}$
as the  set of all real ${m \times n}$ matrices, where $\R$ denotes
the set of real scalars. We denote
$\mathbb{R}^{I_1\times\cdots\times I_N}$  as the set of order $N$
real tensors. Indeed, a matrix is a second order tensor and a vector
is a first order tensor. Note that throughout the paper, tensors are
represented in calligraphic letters like  $\mc{A}$, and the notation
$(\mc{A})_{i_1...i_N}= a_{i_1...i_N}$ represents the scalars.

For a tensor $\mc{A}=(a_{{i_1}...{i_N}{j_1}...{j_M}})
 \in \mathbb{C}^{I_1\times\cdots\times I_N \times J_1 \times\cdots\times J_M},\text{let}~\mc{B}
 =(b_{{i_1}...{i_M}{j_1}...{j_N}})~~ \in $  \nl $
   \mathbb{C}^{J_1\times\cdots\times J_M
\times I_1 \times\cdots\times I_N}$ be the conjugate transpose of
$\mc{A}$, where $b_{{i_1}...{i_M}{j_1}...{j_N}}
=\overline{a}_{{j_1}...{j_M}{i_1}...{i_N}}$. The tensor $\mc{B}$ is
denoted by $\mc{A}^*$.
 When $b_{{i_1}...{i_M}{j_1}...{j_N}} ={a}_{{j_1}...{j_M}{i_1}...{i_N}}$, $\mc{B}$
 is the {\it transpose} of $\mc{A}$, and is denoted by $\mc{A}^T$. Further, a tensor $\mc{O}$ denotes the {\it zero tensor} if  all the entries are zero.
 A tensor $\mc{A}\in
\mathbb{C}^{I_1\times\cdots\times I_N \times I_1 \times\cdots\times
I_N}$ is {\it Hermitian}  if  $\mc{A}=\mc{A}^*$ and {\it
skew-Hermitian} if $\mc{A}= - \mc{A}^*$. Further, a tensor
$\mc{A}\in \mathbb{C}^{I_1\times\cdots\times I_N \times I_1
\times\cdots\times I_N}$  is {\it unitary}  if  $\mc{A}\n
\mc{A}^*=\mc{A}^*\n \mc{A}=\mc{I}$, and {\it idempotent}  if $\mc{A}
\n \mc{A}= \mc{A}.$ In the case of tensors of real entries,
Hermitian, skew-Hermitian and unitary tensors are called {\it
symmetric} (see Definition 3.16, \cite{BraliNT13}), {\it
skew-symmetric} and {\it orthogonal} (see Definition 3.15,
\cite{BraliNT13}) tensors,  respectively. The definition of a
diagonal tensor is borrowed from \cite{sun}, and is obtained by
generalizing Definition 3.12, \cite{BraliNT13}.
\begin{definition} (\cite{sun})\\
A tensor $(\mc{D})_{{i_1}...{i_N}{j_1}...{j_N}}$
  is
   called a {\it diagonal
   tensor} if $d_{{i_1}...{i_N}{j_1}...{j_N}} = 0$ if $(i_1,\cdots,i_N) \neq (j_1,\cdots,j_N).$
\end{definition}

We  recall the definition of an identity  tensor below.

\begin{definition} (Definition 3.13, \cite{BraliNT13}) \\
A tensor
  with entries
     $ (\mc{I})_{i_1i_2 \cdots i_Nj_1j_2\cdots j_N} = \prod_{k=1}^{N} \delta_{i_k j_k}$,
   where
\begin{numcases}
{\delta_{i_kj_k}=}
  1, &  $i_k = j_k$,\nonumber
  \\
  0, & $i_k \neq j_k $.\nonumber
\end{numcases}
 is  called a {\it  unit tensor or identity tensor}.
\end{definition}

We next present the definition of the trace of a tensor which was
introduced earlier in \cite{sun}.

\begin{definition}\label{Tracedefn} (\cite{sun})\\  The trace of a tensor $(\mc{A})_{{i_1}...{i_N}{j_1}...{j_N}}$
  is defined as the sum of the diagonal entries, that is
\begin{equation*}
tr(\mc{A}) = \displaystyle\sum_{i_1 \cdots
i_N}a_{{i_1...i_N}{i_1...i_N}}.
\end{equation*}
\end{definition}

It is denoted by $tr(\mc{A})$. Now, we recall  the singular value
decomposition (SVD) of a tensor which was first introduced in
\cite{BraliNT13} for real tensors, and was then for  a complex
tensors in \cite{sun}.

\begin{lemma}{(Lemma 3.1, \cite{sun})}\\
 A tensor $\mc{A} \in
\mathbb{C}^{I_1\times\cdots\times I_N \times J_1 \times\cdots\times
J_N}$
 can be decomposed  as $$\mc{A} = \mc{U}\n\mc{B}\n\mc{V}^*,$$
 where $\mc{U} \in \mathbb{C}^{I_1\times\cdots\times I_N \times I_1 \times\cdots\times I_N}$ and
 $\mc{V} \in  \mathbb{C}^{J_1\times\cdots\times J_N \times J_1 \times\cdots\times J_N}$ are unitary
 tensors, and
 $\mc{B} \in \mathbb{C}^{I_1\times\cdots\times I_N \times J_1 \times\cdots\times J_N}$ is a
 tensor such that
 $(\mc{B})_{i_1...i_Nj_1...j_N} =0$, if $(i_1,\cdots,i_N) \neq (j_1,\cdots,j_N).$
 \end{lemma}

 Existence and uniqueness of
$\mc{A}^{\dg} $ is shown in Theorem 3.2, \cite{sun}. The authors of
\cite{sun} also showed that $\mc{A}^{\dg} =
\mc{V}\n\mc{B}^{\dg}\n\mc{U}^*$ in the proof of Theorem 3.2,
\cite{sun}. A few properties of $\mc{A}^{\dg}$ are
$(\mc{A}^{\dg})^{\dg}=\mc{A}$ and
$(\mc{A}^{*})^{\dg}=(\mc{A}^{\dg})^{*}$. We conclude this section
with the following lemma  which will be  used in further sections.

\begin{lemma}({Lemma 2.3 $\&$ Lemma 2.6, \cite{bm}})\\\label{revA}
Let $\mc{A}\in \mathbb{C}^{I_1\times\cdots\times I_N \times J_1
\times\cdots\times J_N }$. Then
\begin{enumerate}
\item[(a)] $\mc{A}^* = \mc{A}^{\dg} \n \mc{A} \n \mc{A}^*=\mc{A}^* \n \mc{A} \n \mc{A}^{\dg};$
\item[(b)] $\mc{A} = \mc{A} \n \mc{A}^* \n (\mc{A}^*)^{\dg} = (\mc{A}^*)^{\dg} \n \mc{A}^* \n\mc{A};$
\item[(c)] $\mc{A}^{\dg} = ({\mc{A}^*}\n{\mc{A})^{\dg}}\n\mc{A}^* = {\mc{A}^*}\n(\mc{A}\n\mc{A}^*)^{\dg}.$
\end{enumerate}
 \end{lemma}

\section{Main Results}
In this section, we prove some results concerning the reverse order
law for the Moore-Penrose inverses of  tensors. This section is of
two-fold. Firstly, we obtain some more identities of   the
Moore-Penrose inverses of tensors. Secondly, we present the reverse
order law.

\subsection{Some identities}
\label{subsec1}

When investigating on the reverse order law, we find some
interesting identities. Some of these are used in the next
subsection. The first  part of  the very first result was proved
earlier (see Lemma 2.3 (a), \cite{bm}) using SVD. Here, we have
provided another proof without using SVD.

\begin{theorem}\label{theo11}
Let $\mc{A}\in \mathbb{C}^{I_1\times\cdots\times I_N \times J_1
\times\cdots\times J_N }$. Then,
\begin{itemize}
\item[(a)] $(\mc{A}^*\n\mc{A})^{\dg} = \mc{A}^{\dg} \n(\mc{A}^*)^{\dg}$;
\item[(b)] $(\mc{A}\n\mc{A}^*)^{\dg} = (\mc{A}^*)^{\dg} \n\mc{A}^{\dg}$.
\end{itemize}
\end{theorem}

\begin{proof}$(a)$ By using Definition \ref{defmpi}, we get
$(\mc{A}^*\n\mc{A})\n(\mc{A}^{\dg}\n(\mc{A}^*)^{\dg}) \n (\mc{A}^*\n\mc{A}) =
 \mc{A}^*\n\mc{A}$, $(\mc{A}^{\dg}\n (\mc{A}^*)^{\dg})\n(\mc{A}^*\n\mc{A})\n
 (\mc{A}^{\dg}\n(\mc{A}^*)^{\dg}) = \mc{A}^{\dg}\n(\mc{A}^*)^{\dg}$,
  $ (\mc{A}^*\n\mc{A}\n\mc{A}^{\dg}\n(\mc{A}^*)^{\dg})^* =
   \mc{A}^*\n\mc{A}\n\mc{A}^{\dg}\n(\mc{A}^*)^{\dg}$, and
   $ (\mc{A}^{\dg}\n(\mc{A}^*)^{\dg}\n\mc{A}^*\n\mc{A})^* =
   \mc{A}^{\dg}\n(\mc{A}^*)^{\dg}\n\mc{A}^*\n\mc{A} $. Thus, $(\mc{A}^*\n\mc{A})^{\dg} = \mc{A}^{\dg} \n(\mc{A}^*)^{\dg}$.\\
$(b)$ By replacing $\mc{A}$ by $\mc{A}^*$ in (a), we get the desired
result.
\end{proof}

Recall that a tensor $\mc{B} \in \mathbb{C}^{I_1\times\cdots\times
I_N \times I_1\times\cdots\times I_N }$ is called {\it idempotent}
if $\mc{B}\n \mc{B}=\mc{B}$. The next result presents  a
characterization of an idempotent tensor.

\begin{theorem}
 A tensor $\mc{C} \in \mathbb{C}^{I_1\times\cdots\times I_N \times I_1\times\cdots\times I_N }$ is idempotent if and only if there exists Hermitian and idempotent tensors $\mc{A}$ and $\mc{B}$ in $\mathbb{C}^{I_1\times\cdots\times I_N \times I_1\times\cdots\times I_N }$ such that $\mc{C} = (\mc{B}\n\mc{A})^{\dg}$ in which case $\mc{C}=\mc{A}\n\mc{C}\n\mc{B}$.
\end{theorem}

\begin{proof}
Since $\mc{C}$ is idempotent, we get $ \mc{C} =
(\mc{C}^{\dg}\n\mc{C}\n\mc{C}^{\dg})^{\dg} $,    i.e., $\mc{C} =
(\mc{B}\n\mc{A})^{\dg}$, where $\mc{B}=\mc{C}^{\dg}\n\mc{C}$ and
$\mc{A}=\mc{C}\n\mc{C}^{\dg}$.
Then $\mc{A}\n\mc{C}\n\mc{B}=\mc{C}\n\mc{C}^{\dg}\n\mc{C}\n\mc{C}^{\dg}\n\mc{C}=\mc{C}\n\mc{C}^{\dg}\n\mc{C}=\mc{C}.$\\
Conversely,
\begin{eqnarray*}
\mc{C}&=&(\mc{B}\n\mc{A})^{\dg}\\
& = & (\mc{B}\n\mc{A})^*\n[((\mc{B}\n\mc{A})^{\dg})^{*}\n(\mc{B}\n\mc{A})^{\dg}\n((\mc{B}\n\mc{A})^{\dg})^*]\n(\mc{B}\n\mc{A})^*\\
& = & \mc{A}^*\n\mc{B}^*\n\mc{P}\n\mc{A}^*\n\mc{B}^*, \text{ where } \mc{P}=((\mc{B}\n\mc{A})^{\dg})^{*}\n(\mc{B}\n\mc{A})^{\dg}\n((\mc{B}\n\mc{A})^{\dg})^{*}\\
& = & \mc{A}\n\mc{B}\n\mc{P}\n\mc{A}\n\mc{B}.
\end{eqnarray*}
So,
$\mc{A}\n\mc{C}\n\mc{B}=\mc{A}\n\mc{A}\n\mc{B}\n\mc{P}\n\mc{A}\n\mc{B}\n\mc{B}=\mc{A}\n\mc{B}\n\mc{P}\n\mc{A}\n\mc{B}=\mc{C}$.
Now,
\begin{eqnarray*}
\mc{C}^2  & = & \mc{A}\n\mc{C}\n\mc{B}\n\mc{A}\n\mc{C}\n\mc{B}\\
& = & \mc{A}\n(\mc{B}\n\mc{A})^{\dg}*\mc{B}\n\mc{A}\n(\mc{B}\n\mc{A})^{\dg}\n\mc{B}\\
& = & \mc{A}\n(\mc{B}\n\mc{A})^{\dg}\n\mc{B}\\
& = & \mc{A}\n\mc{C}\n\mc{B}=\mc{C}.
\end{eqnarray*}
So, $\mc{C}$ is idempotent.
\end{proof}

Next result contains three equivalent conditions involving the
Moore-Penrose inverse.
\begin{theorem}
Let $\mc{A}\in \mathbb{C}^{I_1\times\cdots\times I_N \times
I_1\times\cdots\times I_N}$ and $\mc{B} \in
\mathbb{C}^{I_1\times\cdots\times I_N \times I_1\times\cdots\times
I_N}$. Then, the following three conditions are equivalent:
\begin{itemize}
\item[(i)] $\mc{B}\n\mc{A}^{\dg} =\mc{O}$;
\item[(ii)] $\mc{B}\n\mc{A}^* =\mc{O}$;
\item[(iii)] $\mc{B}\n\mc{A}^{\dg}\n\mc{A} =\mc{O}$.
\end{itemize}
\end{theorem}

\begin{proof}
(i)$\Rightarrow$(ii) By post-multiplying $\mc{A}\n \mc{A}^*$ to $\mc{B}\n \mc{A}^{\dg}=\mc{O}$ and then using Lemma \ref{revA} (a), we have $\mc{B}\n \mc{A}^*=\mc{O}$.\\
(ii)$\Rightarrow$(iii) Post-multiplying $\mc{B}\n\mc{A}^* = \mc{O} $
by $(\mc{A}^*)^{\dg}$, we get
$\mc{B}\n\mc{A}^{\dg}\n\mc{A}= \mc{O} $.\\
(iii)$\Rightarrow$(i) Post-multiplying $\mc{B}\n\mc{A}^{\dg}\n\mc{A}
= \mc{O} $ by $\mc{A}^{\dg}$, we obtain $\mc{B}\n\mc{A}^{\dg}=
\mc{O} $.
\end{proof}
A sufficient condition for the commutativity of $\mc{A}$  and
$\mc{A}^{\dg}$ is provided next.
\begin{theorem}
Let $\mc{A}\in \mathbb{C}^{I_1\times\cdots\times I_N \times
I_1\times\cdots\times I_N } $. If $\mc{A}\n\mc{A}^* = \mc{A}^* \n
\mc{A}$, then  $\mc{A}\n\mc{A}^{\dg} = \mc{A}^{\dg} \n \mc{A}$.
\end{theorem}

\begin{proof}
By using Definition \ref{defmpi}, we have
\begin{equation}\label{eq1}
\mc{A}\n\mc{A}^{\dg} =
(\mc{A}^{\dg})^*\n\mc{A}^{\dg}\n\mc{A}\n\mc{A}^*.
\end{equation}
Since $ \mc{A}\n\mc{A}^* = \mc{A}^* \n \mc{A} $, equation
\eqref{eq1} implies
\begin{equation}\label{eq2}
\mc{A}\n\mc{A}^{\dg} = (\mc{A}^*\n\mc{A})^{\dg}\n\mc{A}^*\n\mc{A}.
\end{equation}
By using Theorem \ref{theo11}  (a), equation \eqref{eq2} reduces to
$\mc{A}\n\mc{A}^{\dg} =\mc{A}^{\dg}\n\mc{A}.$
\end{proof}
Converse is not true, and is shown by the following example.

\begin{ex}\label{ex31}
Consider tensors
 $\mc{A} = (a_{ijkl})_{1 \leq i,j,k,l \leq 2}  \in \mathbb{R}^{2\times 2\times 2\times 2}$ such that
\begin{eqnarray*}
a_{ij11} =
    \begin{pmatrix}
    0 & 0\\
    0 & 1\\
    \end{pmatrix},~
a_{ij21} =
    \begin{pmatrix}
    1 & -1\\
    0 & 0\\
    \end{pmatrix},~
a_{ij12} =
    \begin{pmatrix}
    0 & 1\\
    0 & 0\\
    \end{pmatrix} ~~and~~
a_{ij22} =
    \begin{pmatrix}
    1 & 0\\
    -1 & 0\\
    \end{pmatrix}.
\end{eqnarray*}
Then   $\mc{A}^\dg = (x_{ijkl})_{1 \leq i,j,k,l \leq 2}  \in
\mathbb{R}^{2\times 2\times 2\times 2}$ and $\mc{A}^*= (y_{ijkl})_{1
\leq i,j,k,l \leq 2}  \in \mathbb{R}^{2\times 2\times 2\times 2}$,
where
\begin{eqnarray*}
x_{ij11} =
    \begin{pmatrix}
    0 & 1\\
    1 & 0\\
    \end{pmatrix},~
x_{ij21} =
    \begin{pmatrix}
    0 & 1\\
    1 & -1 \\
    \end{pmatrix},~
x_{ij12} =
    \begin{pmatrix}
    0 & 1\\
    0 & 0\\
    \end{pmatrix} ~~and~~
    x_{ij22} =
    \begin{pmatrix}
    1 & 0\\
    0 & 0\\
    \end{pmatrix},
\end{eqnarray*}
and
\begin{eqnarray*}
y_{ij11} =
    \begin{pmatrix}
    0 & 0\\
    1 & 1\\
    \end{pmatrix},~
y_{ij21} =
    \begin{pmatrix}
    0 & 0\\
    0 & -1 \\
    \end{pmatrix},~
y_{ij12} =
    \begin{pmatrix}
    0 & 1\\
    -1 & 0\\
    \end{pmatrix} ~~and~~
    y_{ij22} =
    \begin{pmatrix}
    1 & 0\\
    0 & 0\\
    \end{pmatrix}.
\end{eqnarray*}

We thus have $$\mc{A} \n \mc{A}^\dg=\mc{A}^\dg \n \mc{A},$$  where
\begin{eqnarray*}
(\mc{A} \n \mc{A}^\dg)_{ij11} =
    \begin{pmatrix}
    1 & 0\\
    0 & 0\\
    \end{pmatrix}=(\mc{A}^{\dg} \n \mc{A})_{ij11},~
(\mc{A}^{\dg} \n \mc{A})_{ij21} =
    \begin{pmatrix}
    0 & 1\\
    0 & 0 \\
    \end{pmatrix}=(\mc{A}^{\dg} \n \mc{A})_{ij21},
    \end{eqnarray*}
    \begin{eqnarray*}
(\mc{A} \n \mc{A}^\dg)_{ij12} =
    \begin{pmatrix}
    0 & 0\\
    1 & 0\\
    \end{pmatrix}=(\mc{A}^{\dg} \n \mc{A})_{ij12}, ~~
(\mc{A}\n \mc{A}^\dg)_{ij22} =
    \begin{pmatrix}
    0 & 0\\
    0 & 1\\
    \end{pmatrix}=(\mc{A}^{\dg} \n \mc{A})_{ij22}.
\end{eqnarray*}
But $$\mc{A} \n \mc{A}^* \neq \mc{A}^*\n \mc{A},$$  where

\begin{eqnarray*}
(\mc{A} \n \mc{A}^*)_{ij11} =
    \begin{pmatrix}
    2 & -1\\
    -1 & 0\\
    \end{pmatrix},~
(\mc{A}\n \mc{A}^*)_{ij21} =
    \begin{pmatrix}
    -1 & 2\\
    0 & 0 \\
    \end{pmatrix},
    \end{eqnarray*}
    \begin{eqnarray*}
(\mc{A} \n \mc{A}^*)_{ij12} =
    \begin{pmatrix}
    -1 & 0\\
    1 & 0\\
    \end{pmatrix},~~
(\mc{A} \n \mc{A}^*)_{ij22} =
    \begin{pmatrix}
    0 & 0\\
    0 & 1\\
    \end{pmatrix},
\end{eqnarray*}
and
\begin{eqnarray*}
(\mc{A}^* \n\mc{A})_{ij11} =
    \begin{pmatrix}
    1 & 0\\
    0 & 0\\
    \end{pmatrix},~
(\mc{A}^* \n \mc{A})_{ij21} =
    \begin{pmatrix}
     0 & 1\\
    -1 & 0 \\
    \end{pmatrix},
    \end{eqnarray*}
    \begin{eqnarray*}
(\mc{A}^* \n \mc{A})_{ij12} =
    \begin{pmatrix}
    0 & -1\\
    2 & 1\\
    \end{pmatrix},~~
(\mc{A}^* \n \mc{A})_{ij22} =
    \begin{pmatrix}
    0 & 0\\
    1 & 2\\
    \end{pmatrix}.
\end{eqnarray*}
\end{ex}

We  add a few properties of  the trace of a tensor below. The first
one shows that  the trace of a tensor is a linear mapping.

\begin{lemma}\label{traceproduct}
Let $\mc{A} \in \mathbb{C}^{I_1\times\cdots\times I_N \times
I_1\times\cdots\times I_N}$,   $\mc{B} \in
\mathbb{C}^{I_1\times\cdots\times I_N \times I_1\times\cdots\times
I_N}$ and $\alpha, \beta \in \mathbb{C}$. Then
$$tr(\alpha\mc{A} + \beta\mc{B}) = \alpha tr(\mc{A}) + \beta tr(\mc{B}).$$
\end{lemma}

\begin{proof}
We have
\begin{equation*}
tr(\mc{A}) = \displaystyle\sum_{i_1 \cdots
i_N}a_{{i_1...i_N}{i_1...i_N}} ~~~\textnormal{and}~~~ tr(\mc{B}) =
\displaystyle\sum_{i_1 \cdots i_N}b_{{i_1...i_N}{i_1...i_N}}.
\end{equation*}
So,
\begin{eqnarray*}
tr(\alpha\mc{A} + \beta\mc{B}) & = &  \displaystyle\sum_{i_1... i_N} (\alpha a_{{i_1...i_N}{i_1...i_N}} +\beta b_{{i_1...i_N}{i_1...i_N}}) \\
& = & \alpha \displaystyle\sum_{i_1... i_N}  a_{{i_1...i_N}{i_1...i_N}} + \beta \displaystyle\sum_{i_1... i_N}  b_{{i_1...i_N}{i_1...i_N}} \\
& = & \alpha tr(\mc{A}) + \beta tr(\mc{A}).
\end{eqnarray*}
\end{proof}

Note that for tensors  $\mc{A}  =(a_{{i_1}...{i_N}{j_1}...{j_M}})
\in \mathbb{C}^{I_1\times\cdots\times I_N \times J_1
\times\cdots\times J_M}$  and  $\mc{B}
=(b_{{j_1}...{j_M}{i_1}...{i_N}})\\ \in
\mathbb{C}^{J_1\times\cdots\times J_M \times I_1 \times\cdots\times
I_N}$, we have
\begin{eqnarray*}
tr(\mc{A}\m\mc{B})  &=& \displaystyle\sum_{i_1... i_N} \left(\displaystyle\sum_{k_1...k_M}   a_{{i_1...i_N}{k_1...k_M}}b_{{k_1...k_M}{i_1...i_N}}\right) \\
&=& \displaystyle\sum_{k_1... k_M} \left(\displaystyle\sum_{i_1...i_N}   a_{{i_1...i_N}{k_1...k_M}}b_{{k_1...k_M}{i_1...i_N}}\right) \\
&=& \displaystyle\sum_{k_1... k_M} \left(\displaystyle\sum_{i_1...i_N}  b_{{k_1...k_M}{i_1...i_N}} a_{{i_1...i_N}{k_1...k_M}} \right) \\
&=& tr(\mc{B}\n\mc{A}).
\end{eqnarray*}
Hence,  the tensors in the trace of a product can be switched
without changing the result. This is stated in the next result.

\begin{lemma}\label{traceABBA}
Let $\mc{A}  =(a_{{i_1}...{i_N}{j_1}...{j_M}}) \in
\mathbb{C}^{I_1\times\cdots\times I_N \times J_1 \times\cdots\times
J_M}$  and  $\mc{B}  =(b_{{j_1}...{j_M}{i_1}...{i_N}})\\ \in
\mathbb{C}^{J_1\times\cdots\times J_M \times I_1 \times\cdots\times
I_N}$ be two tensors. Then $tr(\mc{A}\m\mc{B}) =
tr(\mc{B}\n\mc{A})$.
\end{lemma}

Observe that  if $\mc{A}  =(a_{{i_1}...{i_N}{i_1}...{i_N}}), ~
\mc{X}  =(x_{{i_1}...{i_N}{i_1}...{i_N}})  \in
\mathbb{C}^{I_1\times\cdots\times I_N \times I_1 \times\cdots\times
I_N}$, then by Lemma  \ref{traceproduct} and  Lemma \ref{traceABBA},
we can write $tr(\mc{A}\n\mc{X} - \mc{X}\n\mc{A}) = \mc{O}$ but
$tr(\mc{I}) = I_1\times\cdots\times I_N$ and $ I_1\times\cdots\times
I_N \neq 0$, where $\mc{I} \in \mathbb{C}^{I_1\times\cdots\times I_N
\times I_1 \times\cdots\times I_N}$ is the unit tensor. Hence, it is
impossible to find $\mc{X}$ such that $\mc{A}\n\mc{X} -
\mc{X}\n\mc{A} = \mc{I}.$ The next result is the most important tool
for proving the primary result of this paper. The proof uses the
notion of the trace of a tensor.

\begin{lemma}\label{lm1}
Let $\mc{A}\in \mathbb{C}^{I_1\times\cdots\times I_N \times
J_1\times\cdots\times J_M }$. If $\mc{A}^* \n \mc{A}=\mc{O}$, then
$\mc{A}=\mc{O}$.
\end{lemma}

\begin{proof}
Suppose that $ \mc{A}^* \n\mc{A}  =  \mc{O} $, then $ tr(\mc{O}) =
tr(\mc{A}^*\n\mc{A}) $. Let $\mc{C} =\mc{A}^*\n \mc{A}$. Then
\begin{eqnarray*}
c_{i_1...i_Ni_1...i_N}&=&\displaystyle\sum_{k_1...k_M}\overline{a}_{{i_1...i_N}{k_1...k_M}}{a}_{{i_1...i_N}{k_1...k_M}}.
\end{eqnarray*}
Using  Definition  \ref{Tracedefn},  we can now write
\begin{eqnarray*}
0 &  =  & \displaystyle\sum_{i_1... i_N}c_{{i_1...i_N}{i_1...i_N}} \\
 &  =  & \displaystyle\sum_{i_1 ... i_N} \displaystyle\sum_{k_1...k_M}\overline{a}_{{i_1...i_N}{k_1...k_M}}{a}_{{i_1...i_N}{k_1... k_M}}\\
 &  = &  \displaystyle\sum_{i_1 ... i_N} \displaystyle\sum_{k_1...k_M} |a_{{i_1...i_N}{k_1...k_M}}|^2.
\end{eqnarray*}
This implies  each $ |a_{{i_1...i_N}{k_1...k_M}}|^2 =0$ and hence
each $a_{{i_1...i_N}{k_1...k_M}} =0$. Thus $\mc{A} = \mc{O}$.
\end{proof}

We have the following remark obtained by replacing $\mc{A}^*$ in the
place of $\mc{A}$.

\begin{remark}\label{rlm1}
Let $\mc{A}\in \mathbb{C}^{I_1\times\cdots\times I_N \times
J_1\times\cdots\times J_M }$. If $\mc{A} \m \mc{A}^*=\mc{O}$, then
$\mc{A}=\mc{O}$.
\end{remark}

We further  obtain  certain properties of the trace of a tensor
below.

\begin{lemma}
Let $\mc{A} \in \mathbb{C}^{I_1\times\cdots\times I_N \times I_1\times\cdots\times I_N}$ and $\mc{B} \in \mathbb{C}^{I_1\times\cdots\times I_N \times I_1\times\cdots\times I_N}$ be two tensors. Then\\
(a)~$ tr(\mc{A}^\dg\n\mc{B}\n\mc{A}) = tr(\mc{B})$ if  either $\mc{B}\n\mc{A}\n \mc{A}^{\dg}=\mc{B}$ or $\mc{A}\n \mc{A}^{\dg}\n \mc{B}=\mc{B}$; \\
(b)~$tr(\mc{A}^*) = \overline{tr(\mc{A})}$.
\end{lemma}

\begin{proof}
(a) By Lemma \ref{traceABBA}, we obtain
$$tr(\mc{A}^\dg\n\mc{B}\n\mc{A}) = tr(\mc{B}\n\mc{A}\n\mc{A}^\dg) =  tr(\mc{B}),$$ since $\mc{B}\n\mc{A}\n\mc{A}^{\dg}=\mc{B}.$
The proof for the other condition follows similarly.

(b)~ Using  Definition \ref{Tracedefn}, we have
\begin{equation*}
tr(\mc{A}^*) =
\displaystyle\sum_{i_1...i_N}\overline{a}_{{i_1...i_N}{i_1...i_N}} =
\overline{\displaystyle\sum_{i_1...i_N}a_{{i_1...i_N}{i_1...i_N}}} =
\overline{tr(\mc{A})}.
\end{equation*}
\end{proof}

Note that if $\mc{A}$ is invertible in (a) of the above Lemma, then
$ tr(\mc{A}^{-1}\n\mc{B}\n\mc{A}) = tr(\mc{B})$. Clearly, a real
tensor and its transpose have the same trace. Using the linear
property of the trace, it can also be shown that the trace of a
skew-symmetric tensor is zero. The fact $tr(\mc{A})=tr(\mc{A}^T)$
and  Lemma \ref{traceABBA} together yield the following lemma.

\begin{lemma}
If $\mc{A}\in \mathbb{R}^{I_1\times\cdots\times I_N \times
I_1\times\cdots\times I_N}$ and $~\mc{B} \in
\mathbb{R}^{I_1\times\cdots\times I_N \times I_1\times\cdots\times
I_N}$, then $$tr(\mc{A}^T\n\mc{B}) = tr(\mc{B}^T\n\mc{A}) =
tr(\mc{A}\n\mc{B}^T) = tr(\mc{B}\n\mc{A}^T).$$
\end{lemma}

Next result deals with the trace of the the Kronecker product of
tensors. We  first recall the definition of the Kronecker product of
tensors. The Kronecker product $\mc{A}\otimes\mc{B} $ (\cite{sun})
of $\mc{A} \in \mathbb{C}^{I_1\times\cdots\times I_N \times J_1
\times\cdots\times J_N}$
 and $\mc{B} \in \mathbb{C}^{K_1\times\cdots\times K_M \times L_1 \times\cdots\times L_M}$ is defined as $\mc{A}\otimes\mc{B}= (a_{i_1...i_Nj_1...j_N}\mc{B}).$
The result below shows that the trace of the Kronecker product of
two tensors is the product of the traces of the individuals.
\begin{theorem}
Let $\mc{A} =(a_{{i_1}...{i_M}{i_1}...{i_M}}) \in
\mathbb{C}^{I_1\times\cdots\times I_M \times I_1\times\cdots\times
I_M}$ and $\mc{B} =(b_{{j_1}...{j_N}{j_1}...{j_N}}) \in
\mathbb{C}^{J_1\times\cdots\times J_N \times J_1\times\cdots\times
J_N}$ be two tensors. Then $tr(\mc{A}\otimes\mc{B}) =
tr(\mc{A})tr(\mc{B}).$
\end{theorem}

\begin{proof}
The diagonal elements of $(\mc{A}\otimes\mc{B})$ are
$\left(a_{i_1...i_Mi_1...i_M}\mc{B}\right)$ for ${1\leq i_j \leq
I_j} (j=1, \cdots,M)$. Further, the diagonal elements of $\mc{B}$
are  $b_{i_1...i_Ni_1...i_N}$ for ${1\leq i_j \leq I_j} (j=1,
\cdots,N)$. Hence the sum of the diagonal elements of
$\mc{A}\otimes\mc{B}$ is $\left(\displaystyle\sum_{i_1 \cdots
i_M}a_{{i_1...i_M}{i_1...i_M}}\right) \left( \displaystyle\sum_{i_1
\cdots i_N}b_{{i_1...i_N}{i_1...i_N}}\right)= tr(\mc{A})tr(\mc{B})$.
Hence $tr(\mc{A}\otimes\mc{B}) = tr(\mc{A})tr(\mc{B})$.
\end{proof}

We next state that
 the trace of the Einstein product of three tensors is invariant under cyclic permutations, and the proof follows from Lemma \ref{traceABBA}.

\begin{lemma}
Let  $\mc{A} \in \mathbb{C}^{I_1\times\cdots\times I_N \times
I_1\times\cdots\times I_N}$, $\mc{B} \in
\mathbb{C}^{I_1\times\cdots\times I_N \times I_1\times\cdots\times
I_N}$ and  $\mc{C} \in \mathbb{C}^{I_1\times\cdots\times I_N \times
I_1\times\cdots\times I_N}$. Then
$$tr(\mc{A}\n\mc{B}\n\mc{C}) = tr(\mc{B}\n\mc{C}\n\mc{A}) = tr(\mc{C}\n\mc{A}\n\mc{B}).$$
\end{lemma}

Note that arbitrary permutations are not allowed. We next produce an
example which shows this fact.
\begin{example}
Consider tensors $\mc{A} = (a_{ijkl})_{1 \leq i,j,k,l \leq 2}, ~
\mc{B} = (b_{ijkl})_{1 \leq i,j,k,l \leq 2} $ and \\$\mc{C} =
(c_{ijkl})_{1 \leq i,j,k,l \leq 2}  \in \mathbb{R}^{2\times 2\times
2\times 2}$   such that
\begin{eqnarray*}
a_{ij11} =
    \begin{pmatrix}
    0 & 0\\
    1 & 2\\
    \end{pmatrix},~
a_{ij21} =
    \begin{pmatrix}
    1 & 2\\
    -1 & 0\\
    \end{pmatrix},~
a_{ij12} =
    \begin{pmatrix}
    1 & 3\\
    2 & 1\\
    \end{pmatrix} ~~and~~
a_{ij22} =
    \begin{pmatrix}
    0 & 0\\
    0 & 0\\
    \end{pmatrix},
\end{eqnarray*}

\begin{eqnarray*}
b_{ij11} =
    \begin{pmatrix}
    0 & 0\\
    0 & -1\\
    \end{pmatrix},~
b_{ij21} =
    \begin{pmatrix}
    0 & 0\\
    0 & 0\\
    \end{pmatrix},~
b_{ij12} =
    \begin{pmatrix}
    0 & 0\\
    0 & 1\\
    \end{pmatrix} ~~and~~
b_{ij22} =
    \begin{pmatrix}
    0 & 0\\
    0 & 1\\
    \end{pmatrix},
\end{eqnarray*}
and
\begin{eqnarray*}
c_{ij11} =
    \begin{pmatrix}
    1 & -1\\
    2 & 1\\
    \end{pmatrix},~
c_{ij21} =
    \begin{pmatrix}
    1 & 1\\
    1 & 2\\
    \end{pmatrix},~
c_{ij12} =
    \begin{pmatrix}
    0 & 0\\
    0 & 1\\
    \end{pmatrix} ~~and~~
c_{ij22} =
    \begin{pmatrix}
    1 & 3\\
    1 & 2\\
    \end{pmatrix}.
\end{eqnarray*}
Then $\mc{A}\n \mc{B}\n\mc{C}= \mc{O}$,  $ \mc{C}\n \mc{B}\n\mc{A} =
(x_{ijkl})_{1 \leq i,j,k,l \leq 2}  \in \mathbb{R}^{2\times 2\times
2\times 2}$, and $ \mc{B}\n \mc{A}\n\mc{C}\\ = (y_{ijkl})_{1 \leq
i,j,k,l \leq 2}  \in \mathbb{R}^{2\times 2\times 2\times 2}$,  where
\begin{eqnarray*}
x_{ij11} =
    \begin{pmatrix}
    2 & 6\\
    2 & 4\\
    \end{pmatrix},~
x_{ij21} =
\begin{pmatrix}
    1 & 3\\
    1 & 2\\
    \end{pmatrix},~
x_{ij12} =
\begin{pmatrix}
    3 & 9\\
    3 & 6\\
    \end{pmatrix}
     ~~and~~
     x_{ij22} =
    \begin{pmatrix}
    0 & 0\\
    0 & 0\\
    \end{pmatrix}
\end{eqnarray*}
and
\begin{eqnarray*}
y_{ij11} =
    \begin{pmatrix}
    0 & 0\\
    0 & 1\\
    \end{pmatrix},~
y_{ij21} =
    \begin{pmatrix}
    0 & 0\\
    0 & 6\\
    \end{pmatrix},~
y_{ij12} =
    \begin{pmatrix}
    0 & 0\\
    0 & 0\\
    \end{pmatrix} ~~and~~
y_{ij22} =
    \begin{pmatrix}
    0 & 0\\
    0 & 12\\
    \end{pmatrix}.
\end{eqnarray*}
Hence $tr(\mc{A}\n \mc{B}\n\mc{C}) =  0$ but $tr(\mc{C}\n
\mc{B}\n\mc{A}) = tr(\mc{B}\n \mc{A}\n\mc{C}) = 12$.
\end{example}

 To prove the next result, we
define an inner product  of tensors as
\begin{equation*}
\left< \mc{A}, \mc{B}\right> = tr(\mc{A}^*\n\mc{B}) ~~~
\textnormal{for}~~~\mc{A}, \mc{B} \in
\mathbb{C}^{I_1\times\cdots\times I_N\times I_1\times\cdots\times
I_N}.
\end{equation*}
Further, note that the Frobenius norm $||.||_F$  is defined
(\cite{BraliNT13})  as
\begin{equation*}
||\mc{A}||_F = \left( \displaystyle \displaystyle\sum_{i_1...i_N}
|a_{{i_1...i_N}{i_1...i_N}}|^2\right)^{\frac{1}{2}}~~~
\textnormal{for}~~~\mc{A} \in \mathbb{C}^{I_1\times\cdots\times
I_N\times I_1\times\cdots\times I_N}.
\end{equation*}
Thus,   $||\mc{A}||_F^2 = tr(\mc{A}^*\n\mc{A})$. Using the above
definitions, we now prove another result on trace.
\begin{theorem}
Let $\mc{A}\in \mathbb{C}^{I_1\times\cdots\times I_N \times
J_1\times\cdots\times J_N}$ and  $\mc{B} \in
\mathbb{C}^{I_1\times\cdots\times I_N \times J_1\times\cdots\times
J_N}$ be two tensors. Then
$$|tr(\mc{A}^*\n\mc{B})|^2 \leq tr(\mc{A}^*\n\mc{A}) tr(\mc{B}^*\n\mc{B}).$$
\end{theorem}

\begin{proof}
Without loss of generality,  consider $\mc{A} \neq \mc{O}$ and  let
$\alpha = \dfrac{\left< \mc{B}, \mc{A}\right>}{||\mc{A}||^2_F}.$
Then $\left< \alpha \mc{A}-\mc{B}, \mc{A}\right> =\mc{O}$. So
$||\alpha \mc{A}-\mc{B}||^2_F = \left< \alpha \mc{A}-\mc{B},  \alpha
\mc{A}-\mc{B}\right>
= \overline{\alpha} \left< \alpha \mc{A}-\mc{B}, \mc{A}\right> -
\left< \alpha \mc{A}, \mc{B}\right> + \left< \mc{B}, \mc{B}\right>
=\left< \mc{B}, \mc{B}\right> - \left< \alpha \mc{A},
\mc{B}\right>.$ Hence $\alpha \left<  \mc{A}, \mc{B}\right> \leq ||
\mc{B}||_F^2$ as $||\alpha \mc{A}-\mc{B}||^2_F \geq \mc{O}$. Hence
$|\left< \mc{A}, \mc{B}\right>|^2 \leq ||\mc{A}||_F^2||\mc{B}||_F^2$
which  implies $|tr(\mc{A}^*\n\mc{B})|^2 \leq
tr(\mc{A}^*\n\mc{A})tr(\mc{B}^*\n\mc{B})$.
\end{proof}

Since $(\mc{B}\n\mc{A} \n\mc{A}^* -
\mc{C}\n\mc{A}\n\mc{A}^*)\n(\mc{B} - \mc{C})^*=(\mc{B}\n\mc{A} -
\mc{C}\n\mc{A})\n(\mc{B}\n\mc{A} - \mc{C}\n\mc{A})^*$,  we get
$\mc{B}\n\mc{A} = \mc{C}\n\mc{A}$ by using Remark \ref{rlm1}. This
is stated in the next result.

\begin{theorem}
Let $\mc{A} \in \mathbb{C}^{I_1\times\cdots\times I_N \times
I_1\times\cdots\times I_N }, ~~\mc{B}\in
\mathbb{C}^{I_1\times\cdots\times I_N \times I_1\times\cdots\times
I_N }$ and $\mc{C} \in \mathbb{C}^{I_1\times\cdots\times I_N \times
I_1\times\cdots\times I_N }$.
\begin{itemize}
\item[(a)] If $\mc{B}\n\mc{A} \n\mc{A}^* = \mc{C}\n\mc{A}\n\mc{A}^*$, then $\mc{B}\n \mc{A} = \mc{C}\n \mc{A}$.
\item[(b)] If $\mc{B}\n\mc{A}^* \n\mc{A} = \mc{C}\n\mc{A}^*\n\mc{A}$, then $\mc{B}\n \mc{A}^* = \mc{C}\n \mc{A}^*$.
\end{itemize}
\end{theorem}

Sufficient conditions for the Moore-Penrose inverse of the sum of
tensors to be the sum of the Moore-Penrose inverse of the individual
tensors is obtained next.

\begin{theorem}
Let $\mc{A}_i\in \mathbb{C}^{I_1\times\cdots\times I_N \times
I_1\times\cdots\times I_N }$, for all $i$. If $\mc{A}=\sum
\mc{A}_i$, where $\mc{A}_i\n\mc{A}_j^*=\mc{O}$ and $\mc{A}_i^* \n
\mc{A}_j=\mc{O}$,
 whenever $ i\neq j$, then $\mc{A}^{\dg} = \sum \mc{A}_i^{\dg}$.
\end{theorem}

\begin{proof}
Suppose that $\mc{A}=\sum \mc{A}_i$, where
$\mc{A}_i\n\mc{A}_j^*=\mc{O}$ and $\mc{A}_i^* \n \mc{A}_j=\mc{O}$,
whenever $ i\neq j$.
By using Definition \ref{defmpi}, we get $\mc{A}_i\n \mc{A}_j^{\dg} = \mc{A}_i\n\mc{A}_j^*\n(\mc{A}_j^{\dg})^* \n\mc{A}_j^{\dg}$ and $\mc{A}_i^{\dg} \n \mc{A}_j  = \mc{A}_i^{\dg} \n(\mc{A}_i^{\dg})^* \n \mc{A}_i^* \n\mc{A}_j$, which imply $\mc{A}_i\n \mc{A}_j^{\dg} = \mc{O}$ and $\mc{A}_i^{\dg} \n \mc{A}_j  = \mc{O}$ , whenever $ i \neq j $.\\
Let $\mc{B}= \sum \mc{A}_i^{\dg}$. Since $\mc{A}_i\n \mc{A}_j^{\dg}
= \mc{O}$ for $ i\neq j $, we obtain $\mc{A}\n\mc{B}\n\mc{A} = \sum
\mc{A}_i \n \mc{A}_i^{\dg} \n\mc{A}_i$, i.e.,
$\mc{A}\n\mc{B}\n\mc{A} = \mc{A}$.
 Again, the fact $\mc{A}_i^{\dg} \n \mc{A}_j  = \mc{O}$ for $ i\neq j $ gives $\mc{B}\n\mc{A}\n\mc{B} = \sum \mc{A}_i^{\dg} \n \mc{A}_i \n \mc{A}_i^{\dg}$, which yields  $\mc{B}\n\mc{A}\n\mc{B} = \mc{B}$.
 As $\mc{A}_i\n \mc{A}_j^{\dg} = \mc{O}$ for $ i\neq j $, we also have  $(\mc{A}\n\mc{B})^*= \sum \left(\mc{A}_i\n\mc{A}_i^{\dg}\right)^*$, i.e., $(\mc{A}\n\mc{B})^* = \mc{A}\n\mc{B} $,
 and since $\mc{A}_i^{\dg} \n \mc{A}_j  = \mc{O}$ for $ i\neq j $, we get $\left(\mc{B}\n\mc{A}\right)^* = \sum \left({\mc{A}_i^{\dg} \n \mc{A}_i}\right)^*$, i.e., $\left(\mc{B}\n\mc{A}\right)^* = \mc{B}\n\mc{A}$.
Thus, $ \mc{B}= \mc{A}^{\dg} = \sum \mc{A}_i^{\dg} .$
\end{proof}

\begin{lemma}\label{lma 3.8}
Let $\mc{A} \in \mathbb{C}^{I_1\times\cdots\times I_N \times
I_1\times\cdots\times I_N }$ \text{ and } $\mc{B}\in
\mathbb{C}^{I_1\times\cdots\times I_N \times I_1\times\cdots\times
I_N }$. If $\mc{B}$ is invertible, then
\begin{itemize}
\item[(a)]  $(\mc{B}\n \mc{A})^{\dg} \n \mc{B}\n\mc{A}$ = $\mc{A}^{\dg}\n\mc{A}$;
\item[(b)] $\mc{A}\n\mc{B}\n(\mc{A}\n \mc{B})^{\dg} $ = $\mc{A}\n\mc{A}^{\dg}$.
\end{itemize}
\end{lemma}

\begin{proof}

(a) We have  $\mc{B^{\dg}}$=$\mc{B}^{-1}$ as $\mc{B}$ is invertible.
By using Definition \ref{defmpi}, we obtain
$$\mc{B}\n\mc{A}\n(\mc{B}\n\mc{A})^{\dg}\n\mc{B}\n\mc{A} = \mc{B}\n\mc{A},$$
 which on pre-multiplying $\mc{B}^{-1}$ yields $$\mc{A}\n\big(\mc{I}-(\mc{B}\n\mc{A})^{\dg} \n \mc{B}\n\mc{A}\big)\n\mc{A^{\dg}}\n\mc{A}  = \mc{O}.$$
Again, pre-multiplying $\mc{A}^{\dg}$ to the above equality, and using the fact that $\mc{I} - (\mc{B}\n\mc{A})^{\dg} \n \mc{B}\n\mc{A}$ is both idempotent and Hermitian, we obtain $ (\mc{B}\n\mc{A})^{\dg} \n\mc{B} \n\mc{A} = \mc{A}^{\dg} \n\mc{A}$.\\
$(b)$ Proceeding as in $(a)$, one can  have
$\mc{A}\n\mc{B}\n(\mc{A}\n\mc{B})^{\dg} =\mc{A}\n\mc{A}^{\dg}$.
\end{proof}

 The commutativity of $\mc{A}^{\dg}\n\mc{A}$ and $\mc{B}\n\mc{B}^*$ is shown below  under the assumption of a sufficient condition.

\begin{lemma}\label{lm3.13}
Let $\mc{A} \in \mathbb{C}^{I_1\times\cdots\times I_N \times
I_1\times\cdots\times I_N }$ and $\mc{B}\in
\mathbb{C}^{I_1\times\cdots\times I_N \times I_1\times\cdots\times
I_N }$. If $\mc{B}\n\mc{B}^* \n\mc{A}^*=\mc{A}^{\dg} \n \mc{A} \n
\mc{B} \n\mc{B}^* \n\mc{A}^*$, then
$\mc{A}^{\dg}\n\mc{A}\n\mc{B}\n\mc{B}^*=\mc{B}\n\mc{B}^*\n\mc{A}^{\dg}
\n\mc{A}$.
\end{lemma}

\begin{proof}
Suppose that $\mc{B}\n\mc{B}^* \n\mc{A}^*  =  \mc{A}^{\dg} \n\mc{A}
\n\mc{B}\n\mc{B}^* \n\mc{A}^*$, which on post-multiplying
$\mc{A}^{\dg*}$ gives
\begin{equation}\label{eq8}
\mc{B}\n\mc{B}^* \n\mc{A}^{\dg} \n\mc{A}  =  \mc{A}^{\dg} \n\mc{A}
\n\mc{B}\n\mc{B}^* \n\mc{A}^{\dg} \n\mc{A}.
\end{equation}
Since the right side of equation \eqref{eq8} is Hermitian, we so
have
$\mc{A}^{\dg}\n\mc{A}\n\mc{B}\n\mc{B}^*=\mc{B}\n\mc{B}^*\n\mc{A}^{\dg}
\n\mc{A}$.
\end{proof}

Equivalent conditions for the commutativity of
$\mc{A}^{\dg}\n\mc{A}$ and $\mc{B}\n\mc{B}^{\dg}$ is shown below.
\begin{lemma}
Let $\mc{A}\in \mathbb{C}^{I_1\times\cdots\times I_N \times
I_1\times\cdots\times I_N }$ and $\mc{B}\in
\mathbb{C}^{I_1\times\cdots\times I_N \times I_1\times\cdots\times
I_N }$. Then the the commutativity of $\mc{A}^{\dg}\n\mc{A}$ and
$\mc{B}\n\mc{B}^{\dg}$ is equivalent to either of the conditions
\begin{equation}\label{eq9}
\mc{A}^{\dg} \n\mc{A} \n\mc{B}
\n\mc{B}^{\dg}\n\mc{A}^*=\mc{B}\n\mc{B}^{\dg}\n \mc{A}^*
\end{equation}
and
\begin{equation}\label{eq10}
\mc{B}^{\dg} \n\mc{B} \n\mc{A} \n\mc{A}^{\dg} \n\mc{B}^*
=\mc{A}\n\mc{A}^{\dg}\n \mc{B}^* .
\end{equation}
\end{lemma}

\begin{proof}
Clearly,    $\mc{A}^{\dg} \n \mc{A} \n \mc{B} \n \mc{B}^{\dg} \n
\mc{A}^* = \mc{B} \n \mc{B}^{\dg} \n \mc{A}^*$ as $\mc{A}^{\dg} \n
\mc{A}$ and $\mc{B}\n \mc{B}^{\dg}$ are commutative. By
post-multiplying $\mc{A}^{\dg*}$ with equation \eqref{eq9}, we have
$$\mc{A}^{\dg} \n \mc{A} \n \mc{B} \n \mc{B}^{\dg} \n \mc{A}^{\dg} \n\mc{A}  =  \mc{B} \n\mc{B}^{\dg} \n\mc{A}^{\dg} \n \mc{A}.$$
Since, left hand side of the above equality is Hermitian, we so have
$\mc{A}^{\dg} \n \mc{A}\n\mc{B} \n\mc{B}^{\dg} =\mc{B}
\n\mc{B}^{\dg}\n\mc{A}^{\dg} \n \mc{A}$.

In a similar fashion one can prove that, the commutativity of
$\mc{A}^{\dg}\n\mc{A}$ and $\mc{B}\n\mc{B}^{\dg}$ is equivalent to
equation \eqref{eq10}.
\end{proof}

\begin{lemma}
Let $\mc{A} \in \mathbb{C}^{I_1\times\cdots\times I_N \times
J_1\times\cdots\times J_N }$ and $\mc{B}\in
\mathbb{C}^{J_1\times\cdots\times J_N \times I_1\times\cdots\times
I_N }$. Then the conditions
\begin{equation}\label{eq11}
\mc{A}^{\dg} \n\mc{A} \n\mc{B} \n\mc{B}^* \n\mc{A}^*  =
\mc{B}\n\mc{B}^* \n \mc{A}^*
\end{equation}
and
\begin{equation}\label{eq12}
\mc{B} \n\mc{B}^{\dg} \n\mc{A}^* \n\mc{A} \n\mc{B}  =  \mc{A}^*
\n\mc{A}\n\mc{B}
\end{equation}
are equivalent to
\begin{equation}\label{eq13}
(\mc{I} -\mc{A}^{\dg} \n\mc{A})\n\mc{B} \n\mc{B}^* \n\mc{A}^{\dg}
\n\mc{A}  =  \mc{O}
\end{equation}
and \begin{equation}\label{eq14} (\mc{I} -\mc{B}
\n\mc{B}^{\dg})\n\mc{A}^* \n\mc{A} \n\mc{B} \n\mc{B}^{\dg} =
\mc{O}.
\end{equation}

\end{lemma}

\begin{proof}
Post-multiplying equation \eqref{eq11} by $(\mc{A}^{\dg})^*$ and
equation \eqref{eq12} by $\mc{B}^{\dg}$ produce equation
\eqref{eq13} and equation \eqref{eq14}, respectively. Conversely,
post-multiplying equation \eqref{eq13} by $\mc{A}^*$ and equation
\eqref{eq14} by $\mc{B}$ yield equation \eqref{eq11} and equation
\eqref{eq12}, respectively.

\end{proof}

We conclude this subsection with two results which provide different
expressions of $(\mc{A}^*\n\mc{A})^{\dg}$ and
$\mc{A}^{\dg}\n\mc{A}$.
\begin{theorem}
Let $\mc{A}\in \mathbb{C}^{I_1\times\cdots\times I_N \times
J_1\times\cdots\times J_N }$. Then
\begin{equation*}
(\mc{A}^*\n\mc{A})^{\dg} = \mc{A}^{\dg} \n
(\mc{A}\n\mc{A}^*)^{\dg}\n\mc{A} = \mc{A}^* \n
(\mc{A}\n\mc{A}^*)^{\dg}\n(\mc{A}^*)^{\dg}.
\end{equation*}
\end{theorem}
\begin{proof}
By Theorem \ref{theo11}  (a) and Lemma \ref{revA} (c), we have
$(\mc{A}^*\n\mc{A})^{\dg} =
\mc{A}^{\dg}\n(\mc{A}^*)^{\dg}\n\mc{A}^*\n(\mc{A}^*)^{\dg},$ which
implies
\begin{equation}\label{eq15}
(\mc{A}^*\n\mc{A})^{\dg} =
\mc{A}^{\dg}\n(\mc{A}^{\dg})^*\n\mc{A}^*\n(\mc{A}^{\dg})^*,
\end{equation}
i.e., $(\mc{A}^*\n\mc{A})^{\dg} =
\mc{A}^{\dg}\n(\mc{A}^{\dg})^*\n(\mc{A}^{\dg}\n\mc{A})^*.$ By using
Theorem \ref{theo11} (b), we obtain $(\mc{A}^*\n\mc{A})^{\dg} =
\mc{A}^{\dg}\n(\mc{A}\n\mc{A}^*)^{\dg}\n\mc{A}.$ Equation
\eqref{eq15} can be rewritten as $(\mc{A}^*\n\mc{A})^{\dg} =
\mc{A}^{\dg}\n(\mc{A}\n\mc{A}^{\dg})^*\n(\mc{A}^{\dg})^*.$ Then, by
using Theorem \ref{theo11} (b), we obtain
$$(\mc{A}^*\n\mc{A})^{\dg} = \mc{A}^*\n(\mc{A}\n\mc{A}^*)^{\dg}\n(\mc{A}^*)^{\dg}.$$

\end{proof}
\begin{theorem}
Let $\mc{A}\in \mathbb{C}^{I_1\times\cdots\times I_N \times
J_1\times\cdots\times J_N }$. Then
\begin{equation*}
\mc{A}^{\dg}\n\mc{A}= \mc{A}^*\n\mc{A}\n(\mc{A}^*\n\mc{A})^{\dg} =
(\mc{A}^*\n\mc{A})^{\dg}\n\mc{A}^*\n\mc{A}.
\end{equation*}
\end{theorem}
\begin{proof}
By Definition \ref{defmpi}, we have
\begin{equation}\label{eq16}
\mc{A}^{\dg}\n \mc{A} = \mc{A}^*\n \mc{A}
\n\mc{A}^{\dg}\n(\mc{A}^{\dg})^*.
\end{equation}
In view of Theorem \ref{theo11}  (a) equation \eqref{eq16} reduces
to $\mc{A}^{\dg}\n \mc{A} = \mc{A}^*\n \mc{A} \n(\mc{A}^*\n
\mc{A})^{\dg}.$ By Definition \ref{defmpi}, we have
\begin{equation}\label{eq17}
\mc{A}^{\dg}\n \mc{A} = \mc{A}^{\dg}\n (\mc{A}^{\dg})^* \n\mc{A}^*\n
\mc{A}.
\end{equation}
Finally, the application of  Theorem \ref{theo11}  (a) to equation
\eqref{eq17} results
$$\mc{A}^{\dg}\n \mc{A} = (\mc{A}^*\n \mc{A})^{\dg} \n(\mc{A}^*\n \mc{A}).$$
\end{proof}

 \subsection{Reverse Order Law}

In this subsection, we present various necessary and sufficient
conditions for the reverse order law for the Moore-Penrose inverses
of tensors to hold. The first result obtained below deals with the
reverse order law of tensors in the case of one unitary tensor.

\begin{theorem}\label{thm 3.15}
Let $\mc{A}\in \mathbb{C}^{I_1\times\cdots\times I_N \times
I_1\times\cdots\times I_N }$ and $\mc{B}\in
\mathbb{C}^{I_1\times\cdots\times I_N \times I_1\times\cdots\times
I_N }$.
\begin{itemize}
\item[(a)] If $\mc{B}$ is unitary, then $(\mc{A}\n\mc{B})^{\dg} = \mc{B}^*\n\mc{A}^{\dg}.$
\item[(b)] If $\mc{A}$ is unitary, then $(\mc{A}\n\mc{B})^{\dg} = \mc{B}^{\dg}\n\mc{A}^*.$
\end{itemize}
\end{theorem}
\begin{proof}
$(a)$ Since $\mc{B}$ is unitary, $(\mc{A}\n\mc{B})\n
(\mc{B}^*\n\mc{A}^{\dg}) \n(\mc{A}\n\mc{B}) = \mc{A}\n\mc{B},~
(\mc{B}^*\n\mc{A}^{\dg})\n(\mc{A}\n\mc{B})
\n(\mc{B}^*\n\mc{A}^{\dg})\\=(\mc{B}^*\n\mc{A}^{\dg}),~[(\mc{A}\n\mc{B})
\n(\mc{B}^*\n\mc{A}^{\dg})]^* =
(\mc{A}\n\mc{B})\n(\mc{B}^*\n\mc{A}^{\dg}) $,
 and $ [(\mc{B}^*\n\mc{A}^{\dg})\n(\mc{A}\n\mc{B})]^* =
 (\mc{B}^*\n\mc{A}^{\dg})\n(\mc{A}\n\mc{B})$.
 Thus $(\mc{A}\n\mc{B})^{\dg} = \mc{B}^*\n\mc{A}^{\dg}$.\\

$(b)$ In the similar process as  of $(a)$, one can prove $(b)$.
\end{proof}

As a consequence, we have the following result which is stated as
Lemma 2.6 (d) in \cite{bm} .

\begin{theorem}
Let $\mc{A}\in \mathbb{C}^{I_1\times\cdots\times I_N \times
I_1\times\cdots\times I_N },~\mc{B} \in
\mathbb{C}^{I_1\times\cdots\times I_N \times I_1\times\cdots\times
I_N }$ and $\mc{C} \in \mathbb{C}^{I_1\times\cdots\times I_N \times
I_1\times\cdots\times I_N }$. If $\mc{B}$ and $\mc{C}$ are unitary,
then $(\mc{B}\n\mc{A}\n\mc{C})^{\dg} =
\mc{C}^*\n\mc{A}^{\dg}\n\mc{B}^*$.
\end{theorem}

\begin{proof}
The fact $\mc{B}$ is unitary and Theorem \ref{thm 3.15} (b) together
yield
\begin{equation}\label{eq18}
(\mc{B}\n\mc{A}\n\mc{C})^{\dg} = (\mc{A}\n\mc{C})^{\dg}\n\mc{B}^*.
\end{equation}
By using the  fact $\mc{C}$ is unitary and Theorem \ref{thm 3.15}
(a),  equation \eqref{eq18} implies $(\mc{B}\n\mc{A}\n\mc{C})^{\dg}
= \mc{C}^*\n\mc{A}^{\dg}\n\mc{B}^*$.
\end{proof}

The primary result of this paper is presented next which answers the
question posed in the introduction section.

\begin{theorem}\label{rev1}
Let $\mc{A}\in \mathbb{C}^{I_1\times\cdots\times I_N \times
J_1\times\cdots\times J_N }$ and $\mc{B}\in
\mathbb{C}^{J_1\times\cdots\times J_N \times I_1\times\cdots\times
I_N}$. Then $(\mc{A}\n\mc{B})^{\dg} = \mc{B}^{\dg} \n \mc{A}^{\dg}$
if and only if
\begin{equation}\label{eq19}
\mc{A}^{\dg}\n\mc{A}\n\mc{B}\n\mc{B}^* \n\mc{A}^* = \mc{B}\n\mc{B}^*
\n \mc{A}^*
\end{equation}
and
\begin{equation}\label{eq20}
\mc{B}\n\mc{B}^{\dg}\n\mc{A}^*\n\mc{A}\n\mc{B} = \mc{A}^* \n\mc{A}
\n \mc{B}.
\end{equation}
 \end{theorem}

\begin{proof}
Pre-multiplying and post-multiplying equation \eqref{eq19} by
$\mc{B}^{\dg}$ and $((\mc{A}\n\mc{B})^*)^{\dg}$, respectively, we
get
\begin{equation}\label{eq21}
\mc{B}^{\dg} \n\mc{A}^{\dg} \n\mc{A} \n\mc{B}  =
(\mc{A}\n\mc{B})^{\dg} \n \mc{A}\n\mc{B}.
\end{equation}
Taking the conjugate transpose  of equation \eqref{eq20}, we have
\begin{equation}\label{eq22}
\mc{B}^* \n\mc{A}^* \n\mc{A} \n\mc{B} \n\mc{B}^{\dg}  =  \mc{B}^*
\n\mc{A}^* \n\mc{A}
\end{equation}
Pre-multiplying and post-multiplying equation \eqref{eq22} by
$((\mc{A}\n\mc{B})^*)^{\dg}$ and $\mc{A}^{\dg}$, respectively, we
obtain
\begin{equation}\label{eq23}
\mc{A}\n\mc{B} \n\mc{B}^{\dg} \n\mc{A}^{\dg}  =
\mc{A}\n\mc{B}\n(\mc{A}\n\mc{B})^{\dg}.
\end{equation}
In order to show $(\mc{A}\n\mc{B})^{\dg} = \mc{B}^{\dg}
\n\mc{A}^{\dg}$, we have to show that  $\mc{B}^{\dg}\n\mc{A}^{\dg}$
satisfies Definition \ref{defmpi}, and is shown below. Using
equation \eqref{eq21}, we have  $\mc{A}\n\mc{B} \n\mc{B}^{\dg}
\n\mc{A}^{\dg} \n\mc{A} \n\mc{B} = \mc{A}\n\mc{B}$. Applying Lemma
\ref{revA} and Lemma \ref{lm3.13} to equation \eqref{eq19}, we
obtain
$$\mc{B}^*\n\mc{A}^{\dg}\n\mc{A}\n\mc{B}\n\mc{B}^{\dg}\n\mc{A}^*  =
\mc{B}^*\n\mc{A}^*$$  which in turn implies
$$ \mc{B}^{\dg}\n\mc{A}^{\dg}\n\mc{A}\n\mc{B}\n \mc{B}^{\dg}\n\mc{A}^{\dg}  =  \mc{B}^{\dg}\n\mc{A}^{\dg},$$
by pre-multiplication and post-multiplication of
$(\mc{B}^*\n\mc{B})^{\dg}$ and $(\mc{A}\n\mc{A}^*)^{\dg}$,
respectively. Thus, we have $ (\mc{A}\n\mc{B}\n\mc{B}^{\dg}
\n\mc{A}^{\dg})^* = \mc{A} \n\mc{B} \n\mc{B}^{\dg} \n\mc{A}^{\dg} $
by using equation \eqref{eq23} and  $ (\mc{B}^{\dg} \n \mc{A}^{\dg}
\n \mc{A} \n \mc{B})^* = \mc{B}^{\dg} \n \mc{A}^{\dg} \n \mc{A} \n
\mc{B} $, by using equation \eqref{eq21}.

Conversely, in view of Lemma \ref{revA}, $ (\mc{A}\n\mc{B})^{\dg}  =
\mc{B}^{\dg} \n\mc{A}^{\dg}$ implies
\begin{equation}\label{eq24}
\mc{B}^* \n\mc{A}^*  =  \mc{B}^{\dg} \n\mc{A}^{\dg} \n\mc{A}
\n\mc{B} \n\mc{B}^* \n\mc{A}^* .
\end{equation}
Pre-multiplying equation \eqref{eq24} by $\mc{A}\n\mc{B}\n\mc{B}^*
\n\mc{B}$, we get
$$\mc{A} \n\mc{B} \n\mc{B}^* \n(\mc{I} -\mc{A}^{\dg} \n\mc{A})\n\mc{B}\n\mc{B}^* \n\mc{A}^*  =  \mc{O}.$$
Since $(\mc{I} -\mc{A}^{\dg} \n\mc{A})$ is both idempotent and
Hermitian, we obtain $\mc{A}^{\dg} \n\mc{A} \n\mc{B} \n\mc{B}^*
\n\mc{A}^*  =  \mc{B} \n\mc{B}^* \n\mc{A}^*$.
%
Again, by using  Definition \ref{defmpi} and the hypothesis $
(\mc{A}\n\mc{B})^{\dg}  =  \mc{B}^{\dg} \n\mc{A}^{\dg}$, we obtain
\begin{equation}\label{eq25}
\mc{A}\n\mc{B} = \mc{A^{\dg}}^* \n\mc{B^{\dg}}^* \n\mc{B}^*
\n\mc{A}^* \n\mc{A}\n\mc{B}.
\end{equation}
Pre-multiplying equation \eqref{eq25} by $\mc{B}^*
\n\mc{A}^*\n\mc{A}\n\mc{A^*}$, we get
$$\mc{B}^* \n\mc{A}^* \n\mc{A} \n(\mc{I}-\mc{B}\n\mc{B}^{\dg})\n\mc{A}^* \n\mc{A} \n\mc{B}  =  \mc{O}.$$
Since $\mc{I}-\mc{B}\n\mc{B}^{\dg}$ is both idempotent and
Hermitian, we have $\mc{B}\n\mc{B}^* \n\mc{A}^* \n\mc{A}\n\mc{B}  =
\mc{A}^* \n\mc{A} \n\mc{B}$.
\end{proof}

We next replace the conditions in equations (\ref{eq19}) and
(\ref{eq20}) in Theorem \ref{rev1} by another two.

\begin{theorem}\label{Th2}
Let $\mc{A}\in \mathbb{C}^{I_1\times\cdots\times I_N \times
J_1\times\cdots\times J_N }$ and $\mc{B}\in
\mathbb{C}^{J_1\times\cdots\times J_N \times I_1\times\cdots\times
I_N }$.  Then $(\mc{A}\n\mc{B})^{\dg} =\mc{B}^{\dg} \n \mc{A}^{\dg}$
if and only if both $\mc{A}^{\dg} \n\mc{A}\n\mc{B}\n\mc{B}^*$ and
$\mc{A}^* \n\mc{A}\n\mc{B}\n\mc{B}^{\dg}$ are Hermitian.
\end{theorem}

\begin{proof}
Since $\mc{A}^{\dg} \n \mc{A} \n \mc{B} \n \mc{B}^*$ is Hermitian,
we have $ \mc{A}^{\dg} \n \mc{A} \n \mc{B} \n \mc{B}^*  =  \mc{B} \n
\mc{B}^* \n \mc{A}^{\dg} \n \mc{A} $, which on post-multiplication
of  $\mc{A}^*$ yields,
\begin{equation}\label{eq26}
\mc{A}^{\dg} \n \mc{A} \n \mc{B} \n \mc{B}^*\n\mc{A}^* = \mc{B} \n
\mc{B}^*\n\mc{A}^*.
\end{equation}
Again, the fact that  $\mc{A}^* \n \mc{A}\n\mc{B}\n \mc{B}^{\dg}$ is
Hermitian implies  $\mc{B}\n\mc{B}^{\dg} \n \mc{A}^* \n\mc{A} =
\mc{A}^* \n\mc{A}\n\mc{B}\n\mc{B}^{\dg} $, which on post-multiplying
by $\mc{B}$ yields,
\begin{equation}\label{eq27}
\mc{B}\n\mc{B}^{\dg}\n \mc{A}^* \n \mc{A}\n\mc{B}=\mc{A}^* \n \mc{A}
\n \mc{B}.
\end{equation}
Then, by Theorem \ref{rev1}, equations \eqref{eq26} and \eqref{eq27}
imply $(\mc{A}\n\mc{B})^{\dg} =\mc{B}^{\dg} \n \mc{A}^{\dg}$.

Conversely, suppose that $(\mc{A}\n\mc{B})^{\dg} =\mc{B}^{\dg} \n
\mc{A}^{\dg}$. By Theorem \ref{rev1},  we have
\begin{equation}\label{eq28}
\mc{A}^{\dg}\n\mc{A}\n\mc{B}\n\mc{B}^*\n\mc{A}^* =
\mc{B}\n\mc{B}^*\n\mc{A}^*
\end{equation}
and
\begin{equation}\label{eq29}
\mc{B}\n\mc{B}^{\dg}\n\mc{A}^*\n\mc{A}\n\mc{B}
=\mc{A}^*\n\mc{A}\n\mc{B}.
\end{equation}
Post-multiplying equation \eqref{eq28} by $(\mc{A}^*)^{\dg}$, we
have
$\mc{A}^{\dg} \n \mc{A}\n \mc{B} \n \mc{B}^*$ is Hermitian
and post-multiplying equation \eqref{eq29} by $\mc{B}^{\dg}$, we get
$\mc{A}^*\n\mc{A}\n\mc{B}\n\mc{B}^{\dg}$ is Hermitian.
\end{proof}

The followings are straightforward consequences of the above result.

\begin{remark}
$\mc{A}^{\dg} \n \mc{A}\n\mc{B}\n\mc{B}^{\dg}$  is Hermitian is
equivalent to the fact that $\mc{A}^{\dg} \n \mc{A}$ and
$\mc{B}\n\mc{B}^*$ commutes, and
$\mc{A}^*\n\mc{A}\n\mc{B}\n\mc{B}^{\dg}$ is Hermitian is equivalent
to the fact that $\mc{A}^*\n\mc{A}$ and $\mc{B}\n\mc{B}^{\dg}$
commutes.
\end{remark}

The next result provides only one equivalent condition for the
reverse order law instead of two as in earlier results.

\begin{theorem}
Let $\mc{A}\in \mathbb{C}^{I_1\times\cdots\times I_N \times
J_1\times\cdots\times J_N }$ and $\mc{B}\in
\mathbb{C}^{J_1\times\cdots\times J_N \times I_1\times\cdots\times
I_N }$.  Then $(\mc{A} \n \mc{B})^{\dg} = \mc{B}^{\dg} \n
\mc{A}^{\dg}$ if and only if
\begin{equation}\label{eq30}
\mc{A}^{\dg} \n \mc{A} \n \mc{B} \n \mc{B}^* \n\mc{A}^* \n \mc{A} \n
\mc{B} \n \mc{B}^{\dg} = \mc{B} \n \mc{B}^* \n\mc{A}^* \n \mc{A}.
\end{equation}
\end{theorem}

\begin{proof}
Pre-multiplying and post-multiplying equation \eqref{eq30} by
$\mc{A}^{\dg} \n  \mc{A}$ and $\mc{A}^{\dg}$, respectively, we get
\begin{equation}\label{eq31}
\mc{A}^{\dg} \n \mc{A} \n \mc{B} \n \mc{B}^*\n\mc{A}^* = \mc{B} \n
\mc{B}^*\n\mc{A}^*.
\end{equation}
Again, pre-multiplying and post-multiplying equation \eqref{eq30} by
$\mc{B}^{\dg}$ and $\mc{B} \n \mc{B}^{\dg}$, respectively, we have
\begin{equation}\label{eq32}
\mc{B}\n\mc{B}^{\dg}\n \mc{A}^* \n \mc{A}\n\mc{B}=\mc{A}^* \n \mc{A}
\n \mc{B}.
\end{equation}
Then,  equations \eqref{eq31} and \eqref{eq32} together imply
$(\mc{A}\n\mc{B})^{\dg} = \mc{B}^{\dg} \n\mc{A}^{\dg}$ by Theorem
\ref{rev1}.

Conversely, the fact $(\mc{A}\n\mc{B})^{\dg} = \mc{B}^{\dg}
\n\mc{A}^{\dg}$  is equivalent to  equations \eqref{eq19} and
\eqref{eq20}, by Theorem \ref{rev1}. Post-multiplying equation
\eqref{eq19} by $\mc{A}$ and using equation \eqref{eq20}, we obtain
$$\mc{A}^{\dg} \n \mc{A} \n \mc{B} \n \mc{B}^* \n\mc{A}^* \n \mc{A} \n \mc{B} \n \mc{B}^{\dg} = \mc{B} \n \mc{B}^* \n\mc{A}^* \n \mc{A}.$$
\end{proof}

We next present another characterization of the reverse order law.

\begin{theorem}\label{theorem4}
Let  $\mc{A}\in \mathbb{C}^{I_1\times\cdots\times I_N \times
J_1\times\cdots\times J_N }$ and $\mc{B}\in
\mathbb{C}^{J_1\times\cdots\times J_N \times I_1\times\cdots\times
I_N }$.  Then $(\mc{A}\n\mc{B})^{\dg} =\mc{B}^{\dg} \n \mc{A}^{\dg}$
if and only if both the equations
\begin{equation}\label{eq35}
\mc{A}^{\dg} \n \mc{A} \n \mc{B} = \mc{B}\n(\mc{A}\n\mc{B})^{\dg} \n
\mc{A} \n \mc{B}
\end{equation}
and
\begin{equation}\label{eq36}
\mc{B} \n \mc{B}^{\dg} \n \mc{A}^* = \mc{A}^* \n\mc{A}\n\mc{B} \n
(\mc{A} \n \mc{B})^{\dg}
\end{equation}
are satisfied.
\end{theorem}

\begin{proof}
By Theorem \ref{rev1}, equations \eqref{eq19} and \eqref{eq20} hold
true. Now, post-multiplying $((\mc{A}\n\mc{B})^{\dg})^*$ to equation
\eqref{eq19}, we have $\mc{A}^{\dg} \n \mc{A} \n \mc{B}\n
(\mc{A}\n\mc{B})^* \n ((\mc{A}\n\mc{B})^{\dg})^*=
\mc{B}\n(\mc{A}\n\mc{B})^*\n ((\mc{A}\n\mc{B})^{\dg})^*$ which
yields $\mc{A}^{\dg} \n  \mc{A}\n\mc{B} \n (\mc{A}\n\mc{B})^{\dg} \n
\mc{A}\n\mc{B}= \mc{B}\n (\mc{A}\n\mc{B})^{\dg}\n \mc{A}\n\mc{B}.$
Post-multiplying $(\mc{A}\n\mc{B})^{\dg}$ to $(\mc{B} \n
\mc{B}^{\dg})^* \n \mc{A}^*\n \mc{A}\n\mc{B}= \mc{A}^* \n
\mc{A}\n\mc{B}$,  which follows from \eqref{eq20}, we obtain
\begin{eqnarray*}
\mc{A}^*\n\mc{A}\n\mc{B}\n (\mc{A}\n\mc{B})^{\dg} & = & (\mc{B}^{\dg})^*\n(\mc{A}\n\mc{B})^*\n\mc{A}\n\mc{B}\n(\mc{A}\n\mc{B})^{\dg} \\
& = & (\mc{B}^{\dg})^*\n (\mc{A}\n\mc{B})^*\\
& = & (\mc{B}^{\dg})^*\n \mc{B}^*\n\mc{A}^*\\
& = & \mc{B}\n \mc{B}^{\dg}\n\mc{A}^*.
\end{eqnarray*}

Conversely, post-multiplying $(\mc{A}\n\mc{B})^*$ to equation
\eqref{eq35}, we have
\begin{eqnarray*}
\mc{A}^{\dg} \n \mc{A} \n \mc{B} \n \mc{B}^*\n\mc{A}^*& = &\mc{B}\n (\mc{A}\n\mc{B})^{\dg} \n  \mc{A} \n \mc{B}\n (\mc{A}\n\mc{B})^*\\
& = &\mc{B}\n (\mc{A}\n\mc{B})^*\\
& = &\mc{B}\n \mc{B}^*\n\mc{A}^*.
\end{eqnarray*}
Now, post-multiplying $\mc{A}\n\mc{B}$ to equation \eqref{eq36}, we
obtain
$$
\mc{B} \n  \mc{B}^{\dg} \n \mc{A}^*\n \mc{A}\n\mc{B} =
 \mc{A}^* \n \mc{A}\n\mc{B} \n  (\mc{A} \n \mc{B})^{\dg} \n \mc{A}\n\mc{B}=
 \mc{A}^* \n \mc{A}\n\mc{B}.$$
Hence,  $(\mc{A}\n\mc{B})^{\dg} =\mc{B}^{\dg} \n \mc{A}^{\dg}$ by
Theorem \ref{rev1}.
\end{proof}

The last result of this paper provides a sufficient condition for
the reverse order law.

\begin{theorem}
Let  $\mc{A}\in \mathbb{C}^{I_1\times\cdots\times I_N \times
J_1\times\cdots\times J_N }$ and $\mc{B}\in
\mathbb{C}^{J_1\times\cdots\times J_N \times I_1\times\cdots\times
I_N }$. If $(\mc{A} \n \mc{B})^{\dg} = \mc{B}^{\dg} \n \mc{A}^{\dg}
$, then  $\mc{A}^{\dg} \n \mc{A} $ and $\mc{B} \n \mc{B}^{\dg} $
commute.
\end{theorem}

\begin{proof}
By Theorem \ref{theorem4},  we have $\mc{A}^{\dg} \n \mc{A} \n
\mc{B} = \mc{B}\n\mc{B}^{\dg} \n\mc{A}^{\dg} \n \mc{A} \n \mc{B}.$
Post-multiplying $\mc{B}^{\dg}$ and taking  conjugate transpose in
both sides,  we get
$$(\mc{A}^{\dg} \n  \mc{A} \n \mc{B}\n \mc{B}^{\dg})^*= (\mc{B}\n\mc{B}^{\dg} \n\mc{A}^{\dg} \n \mc{A} \n \mc{B} \n \mc{B}^{\dg}),$$ i.e.,
 $\mc{B}\n \mc{B}^{\dg} \n \mc{A}^{\dg} \n \mc{A} = \mc{B}\n\mc{B}^{\dg} \n\mc{A}^{\dg} \n \mc{A} \n \mc{B} \n \mc{B}^{\dg}$.
So $\mc{A}^{\dg} \n  \mc{A} $ and $\mc{B} \n  \mc{B}^{\dg} $
commute.

\end{proof}

Converse of the above theorem is  not true, and is shown by the
following example.

\begin{example}
Consider tensors $\mc{A} = (a_{ij})_{1 \leq i,j \leq 2}  \in
\mathbb{R}^{ 2\times 2}$  and  $\mc{B} = (b_{ijkl})_{1 \leq i,j,k,l
\leq 2}  \in \mathbb{R}^{2\times 2\times 2\times 2}$ such that
    \begin{equation*}
 a_{ij} =
    \begin{pmatrix}
    1 & 1 \\
    1 & 0\\
    \end{pmatrix},
    \end{equation*}
    and
\begin{eqnarray*}
b_{ij11} =
    \begin{pmatrix}
    0 & 0\\
    0 & 1\\
    \end{pmatrix},~
b_{ij21} =
    \begin{pmatrix}
    1 & -1\\
    0 & 0\\
    \end{pmatrix},~
b_{ij12} =
    \begin{pmatrix}
    0 & 1\\
    0 & 0\\
    \end{pmatrix} ~~and~~
b_{ij22} =
    \begin{pmatrix}
    1 & 0\\
    -1 & 0\\
    \end{pmatrix}.
\end{eqnarray*}
Then  $\mc{A}^\dg= (x_{ij})_{1 \leq i,j\leq 2}  \in
\mathbb{R}^{2\times 2}$, and  $\mc{B}^\dg = (y_{ijkl})_{1 \leq
i,j,k,l \leq 2}  \in \mathbb{R}^{2\times 2\times 2\times 2}$, where
\begin{eqnarray*}
x_{ij} =
    \begin{pmatrix}
    0 & 1\\
    1 & -1\\
    \end{pmatrix},
\end{eqnarray*}

and

\begin{eqnarray*}
y_{ij11} =
    \begin{pmatrix}
    0 & 1\\
    1 & 0\\
    \end{pmatrix},~
y_{ij21} =
    \begin{pmatrix}
    0 & 1\\
    1 & -1 \\
    \end{pmatrix},~
y_{ij12} =
    \begin{pmatrix}
    0 & 1\\
    0 & 0\\
    \end{pmatrix} ~~and~~
    y_{ij22} =
    \begin{pmatrix}
    1 & 0\\
    0 & 0\\
    \end{pmatrix}.
\end{eqnarray*}
We thus have $$\mc{A}^\dg \n \mc{A} \n\mc{B}\n\mc{B}^\dg
=\mc{B}\n\mc{B}^\dg \n \mc{A}^\dg \n \mc{A}$$ as  $\mc{A}^\dg
\n\mc{A} = \mc{I}$. But $(\mc{A}\n\mc{B})^\dg =
\begin{pmatrix}
    0 & 0\\
    1 & 0\\
\end{pmatrix}$
and $\mc{B}^\dg \n \mc{A}^\dg =
\begin{pmatrix}
    -1 & 2\\
    1 & -1\\
\end{pmatrix}$.
Hence $$ (\mc{A}\n\mc{B})^\dg \neq \mc{B}^\dg \n \mc{A}^\dg.$$
\end{example}

\noindent {\small {\bf Acknowledgments.}\\
 The last author acknowledges the support provided by Science
  and Engineering Research Board, Department of Science and Technology,
  New Delhi, India, under the grant number YSS / 2015 / 000303. }

\bibliographystyle{amsplain}

\end{document}